\documentclass[12pt, reqno]{amsart}
\usepackage{amsmath,amsfonts,amssymb}
\usepackage{etoolbox}
\usepackage{comment}
\usepackage{mathrsfs}
\usepackage{longtable}
\usepackage{graphicx}
\usepackage{mathtools}
\usepackage[breaklinks]{hyperref}
\usepackage{graphicx}
\usepackage{xcolor}  
\input amssym.def
\usepackage{amscd}
\usepackage[mathscr]{eucal}

\usepackage{tikz}

\theoremstyle{plain}
\newtheorem{thm}{Theorem}[section]

\newtheorem{lem}{Lemma}[section]

\newtheorem{prop}{Proposition}[section]

\newtheorem{rem}{Remark}[section]

\newcommand{\sgn}{\operatorname{sgn}}

\newcommand{\Z}{\mathbb{Z}}
\newcommand{\N}{\mathbb{N}}

\renewcommand{\H}{\mathbb{H}}
\numberwithin{equation}{section}

\renewcommand{\Im}{{\mathrm Im \,}}
\begin{document}

\title[On periodic sign changes]{ON  PERIODIC SIGN CHANGES FOR weighted representations of integers as colored sums of triangular and generalized pentagonal numbers}

\author{Ben Kane}
\address{Department of Mathematics, The University of Hong Kong, Pokfulam, Hong Kong}
\email{bkane@hku.hk}

\author{Meenu Sharma}
\address{Department of Mathematics, The University of Hong Kong, Pokfulam, Hong Kong }
\email{sharmameenuphd@gmail.com}

\thanks{The research of the first author was supported by grants from the Research Grants Council of the Hong Kong SAR, China (project numbers HKU 17314122, HKU 17305923). Part of the research contained in this paper was completed while the second author was a Ph.D. student in Kerala School of Mathematics}

\date{\today}

\begin{abstract}
In this paper, we study sign changes and vanishing for the number of representations of an integer as a sum of triangular numbers plus three-colored sums of generalized pentagonal numbers with an even number of parts minus those with an odd number of parts. We study this by investigating coefficients appearing in the $q$-series expansion of $F(z) = \frac{\eta(z)}{{\eta(2z)}^2 {\eta(3z)}^3}$, where $\eta$ is the Dedekind-eta function. We use the Hardy-Ramanujan-Rademacher circle method to give an asymptotic formula for the coefficients.
\end{abstract}

\maketitle

\section{Introduction and statement of result}
Let $T_k :=\frac{k(k+1)}{2}$ be the $k$-th triangular number and  $P_5(m):=(3 m^2-m)/2$ be the $m$-th generalized pentagonal number (with $m\in\Z$). Consider the set $\Lambda_{j_1,j_2,j_3,j_4}(n)$ of $4$-colored partitions of $n$ as a sum 
\[
n=n_1+3{\color{blue}{n_2}}+3{\color{red}{n_3}}+3{\color{green}n_4},
\]
with
\[
n_1=\sum_{j=1}^{j_1} T_{x_j}
\]
(with $x_j\geq 1$) is a sum of exactly $j_1$ triangular numbers, 
\[
n_2=\sum_{j=1}^{j_2} P_5(y_j)
\]
(with $y_j\neq 0$) is a sum of exactly $j_2$ generalized pentagonal numbers, 
\[
n_3=\sum_{j=1}^{j_3} P_5(w_j)
\]
(with $w_j\neq 0$) is a sum of exactly $j_3$ generalized pentagonal numbers, and 
\[
n_4=\sum_{j=1}^{j_4} P_5(z_j)
\]
(with $z_j\neq 0$) is a sum of exactly $j_4$ generalized pentagonal numbers. For each $\lambda\in \Lambda_{j_1,j_2,j_3,j_4}(n)$, we define a weighting $\omega_{\lambda}$ by $1$ if the parity of the number of parts equals the parity of  
the sum of all of the "side lengths" (i.e., $m$ for $P_5(m)$) of the generalized pentagonal numbers and $-1$ if they are not equal.\footnote{For example, in the partition $396=T_2+T_5+3{\color{blue}P_{5}(4)}+3({\color{red}P_5(3)+P_5(-3)})+3{\color{green}P_5(-7)}$, there are $6$ parts and there are 3 odd "side lengths", $3$, $-3$, and $-7$, so the weighting is $-1$.}
Explicitly, we have 
\[
\omega_{\lambda}:=(-1)^{j_1+j_2+j_3+j_4}(-1)^{y_1+y_2+\dots+y_{j_2}+w_{1}+\dots+w_{j_3}+z_1+\dots+z_{j_4}}.
\]
We let 
\[
r_{\operatorname{e}}(n):=\#\{\lambda\in \Lambda_{j_1,j_2,j_3,j_4}(n):\omega_{\lambda}=1\}
\]
be the number of elements with the same parities and 
\[
r_{\operatorname{o}}(n):=\#\{\lambda\in \Lambda_{j_1,j_2,j_3,j_4}(n):\omega_{\lambda}=-1\}
\]
be the number of such partitions with opposite parities.
\begin{thm}\label{thm:combinatorics}
 We have
\begin{align*}
            r_{\operatorname{e}}(n)&>r_{\operatorname{o}}(n)\qquad \text{if}\; n\equiv 0, 2, 3 \pmod 6,\\
            r_{\operatorname{e}}(n)&<r_{\operatorname{o}}(n)\qquad \text{if}\; n\equiv 1, 5 \pmod 6,\\
            r_{\operatorname{e}}(n)&=r_{\operatorname{o}}(n)\qquad \text{if}\; n\equiv 4 \pmod 6.
        \end{align*}
\end{thm}
In order to prove Theorem \ref{thm:combinatorics}, we investigate the generating function for the difference, which turns out to be infinite product related to an $\eta$-quotient. Consider a positive integer $m$ and a sequence of integers $\kappa_1, \kappa_2,...,\kappa_m$. We begin by defining an infinite product:
\begin{equation}\label{eqn:Gkappadef} 
G_{\kappa}(q) := \prod_{k=1}^{m}(q^k;q^k)^{\kappa_k}_\infty= \prod_{k=1}^m \prod_{n=1}^{\infty}\left(1-q^{nk}\right)^{\kappa_k}.
\end{equation}
Here, $(a;q)_\infty = \prod_{j=0}^{\infty}(1-aq^j)$ is the standard infinite 
$q-$Pochhammer symbol. This product can be expanded as a power series in $q:$
\[
G_{\kappa}(q) = \sum_{n \geq 0}A_{\kappa}(q) q^n.
\]
The sign sequence, denoted by $\sgn(A_{\kappa}(n))$ for different $\kappa$  has been a subject of significant investigation in the literature, e.g. \cite{Andrew}, \cite{Daniel}, \cite{Sch & Zhou}, \cite{Ben}. Theorem \ref{thm:combinatorics} is equivalent to a single case of a conjecture formulated in \cite{Ben} for purely periodic sign changes $\pmod{6}$. Let $\tau \in \H$, define
\[F(\tau) = \frac{\eta(\tau)}{{\eta(2\tau)}^2 {\eta(3\tau)}^3}\]
where $\eta(\tau)$ denotes the Dedekind eta function:
\[\eta(\tau) = q^{\frac{1}{24}} \prod_{n = 1}^{\infty} (1 - q^n), \;\;\;\;\;\;\; q = e^{2\pi i \tau}.\]
One sees directly that 
\[
F(\tau)= q^{-\frac{1}{2}}G_{(1,-2,-3)}(q),
\]
i.e., $F$ is essentially \eqref{eqn:Gkappadef} for $\kappa_1=1$, $\kappa_2=-2$, and $\kappa_3=-3$. We abbreviate $G:=G_{(1,-2,-3)}$ throughout and write 
\[
G(q)=:\sum_{n\geq 0} c(n) q^n.
\]
We show in Section \ref{sec:combinatorics} that $c(n)$ counts the difference $r_{\operatorname{e}}(n)-r_{\operatorname{o}}(n)$, so Theorem \ref{thm:combinatorics} then follows from the following result. 
\begin{thm}\label{thm 1}
    We have
    \begin{align*}
        \sgn(c(n)) =
        \begin{cases}
            1 & \text{if}\; n\equiv 0, 2, 3 \pmod 6,\\
            -1 & \text{if}\; n\equiv 1, 5 \pmod 6,\\
            0 & \text{if}\; n\equiv 4 \pmod 6.
        \end{cases}
    \end{align*}
\end{thm}
In order to investigate $\sgn(c(n))$, one writes $G$ as a ratio of partition-generating functions and then uses the growth of the function to obtain an exact formula via the Circle Method. These methods were established by Hardy and Ramanujan \cite{H & R} to obtain an asymptotic formula for the number of partitions $p(n)$ of $n\in\N$ and extended by Rademacher to obtain an exact formula for $p(n)$ (see \eqref{part-asym}).

The paper is organized as follows. In Section \ref{sec:prelims}, we introduce weakly holomorphic modular forms and recall the modular properties of the partition generating function and some inequalities that are needed to bound the error from the Circle Method. In Section \ref{sec:mainproof}, we prove Theorem \ref{thm 1} for $n\not\equiv 4\pmod{6}$, and in Section \ref{sec:vanish} we establish the vanishing property for $n\equiv 4\pmod{6}$. Finally, we establish the connection with colored representations by triangular numbers and generalized pentagonal numbers in Section \ref{sec:combinatorics}, establishing Theorem \ref{thm:combinatorics}.

\section{Preliminaries}\label{sec:prelims}

\subsection{Weakly holomorphic modular forms}

Let $\Gamma \subseteq \mathrm{SL}_{2}(\mathbb{Z})$ be a congruence subgroup and let 
$\chi \colon \Gamma \to \mathbb{C}^{\times}$ be a multiplier system of weight $\kappa \in \mathbb{Z}$.
A function $F \colon \mathbb{H} \to \mathbb{C}$ is called a \emph{weakly holomorphic modular form of 
weight $\kappa$ on $\Gamma$ with multiplier $\chi$} if it satisfies the following three conditions.\\
\textbf{1. Modularity:} For every $\gamma = \left(\begin{smallmatrix}a & b \\ c & d\end{smallmatrix}\right) \in \Gamma$,
    \[
    F \big|_{\chi} \gamma(\tau) = \chi(\gamma) F(\tau),
    \]
    where the weight-$\kappa$ slash operator is defined by
    \[
    F \big|_{\chi} \gamma(\tau) := (c\tau + d)^{-\kappa} F(\gamma\tau), \qquad \gamma\tau = \frac{a\tau + b}{c\tau + d}.
    \]
\textbf{2. Holomorphy on $\mathbb{H}$:} $F$ is holomorphic on the upper half-plane $\mathbb{H}$.\\
\textbf{3. Meromorphy at the cusps:} For each cusp $\varrho$ of $\Gamma$, choose a matrix 
    $M_{\varrho} \in \mathrm{SL}_{2}(\mathbb{Z})$ with $M_{\varrho}(\infty) = \varrho$ and set 
    $F_{\varrho} := F \big|_{\kappa} M_{\varrho}$. Then $F_{\varrho}$ is periodic with some 
    period $\sigma_{\varrho} \in \mathbb{N}$ and therefore admits a Fourier expansion of the form
    \[
    F_{\varrho}(\tau) = \sum_{n \in \mathbb{Z}} c_{F,\varrho}(n) \, q^{n/\sigma_{\varrho}}, \qquad q := e^{2\pi i \tau}.
    \]
    We require that for every cusp $\varrho$ there exists a constant $n_{0} = n_{0}(\varrho) \in \mathbb{Z}$ such that 
    $c_{F,\varrho}(n) = 0$ for all $n < n_{0}$. In other words, $F$ has at most a pole at each cusp, i.e., 
    only finitely many negative-exponent terms appear in its Fourier expansion there.\\
If the cusp is $\infty$ (so that $M_{\varrho}$ is the identity), we simply write the expansion as
\[
F(\tau) = \sum_{n \gg -\infty} c_{F}(n) q^{n}.
\]
The finite collection of terms with $n < 0$ in the expansion of $F_{\varrho}$ is called the 
\emph{principal part} of $F$ at the cusp $\varrho$. A weakly holomorphic modular form whose principal part is trivial at all cusps is called a \emph{holomorphic modular form}; if it also 
vanishes at every cusp, then we call it a \emph{cusp form}.

\subsection{The partition generating function}
As noted in the introduction, $G(z)$ can be written as a ratio of partition-generating functions. The generating function for partitions, denoted as $P(q)$, is expressed as
\[
P(q) = \sum_{n\geq 0} p(n) q^n = \prod_{n\geq 1}\frac{1}{1-q^n}.
\]   
Thus we can write 

\begin{equation}\label{a}
q^{\frac{1}{2}}F(\tau)= G(q) = \frac{P^2(q^2) P^3(q^3)}{P(q)} = \sum_{n\geq 0} c(n) q^n.      
\end{equation}
For $k\in\N$ and $h\in\Z$ with $\gcd(h,k)=1$, consider $\tau=\frac{h}{k}+\frac{iz}{k^2}$ with $\mathfrak{R}(z)>0$ and let $h'\in\Z$ be chosen so that $hh'\equiv -1\pmod{k}$. Hardy and Ramanujan \cite[Lemma 4.31]{H & R} showed that the partition function satisfies the modular relation
\begin{equation}\label{eqn:Pmodular}
P\left(e^{\frac{2\pi i h}{k} - \frac{2\pi z}{k^2}}\right) = \omega_{h, k} ({\frac{z}{k}})^{\frac{1}{2}} \exp\left({\frac{\pi}{12z} - \frac{\pi z}{12k^2}}\right) P\left(e^{\frac{2\pi i h'}{k} - \frac{2\pi}{z}}\right),
\end{equation}
where 
\begin{equation}\label{omega}
  \omega_{h,k}=e^{\pi i s(h,k)},
\end{equation}
and $s(h,k)$ is the \begin{it}Dedekind sum\end{it} defined by
\begin{equation}\label{Dedsum}
  s(h,k)=\sum_{1\le j<k}\bigg(\frac{j}{k}-
    \left\lfloor\frac{j}{k}\right\rfloor-
    \frac{1}{2}\bigg)\bigg(\frac{jh}{k}-
    \left\lfloor\frac{jh}{k}\right\rfloor-\frac{1}{2}\bigg).
\end{equation}
Using the modular properties \eqref{eqn:Pmodular} of $P(q)$, Hardy and Ramanujan \cite{H & R} developed the Circle Method to obtain an asymptotic formula for $p(n)$ which was later improved by Rademacher \cite{R partition} to the exact formula
\begin{equation}\label{part-asym}
  p(n)=\frac{1}{\pi\sqrt{2}}
  \sum_{k\ge 0}k^{1/2}\sum_{\substack{h\!\!\pmod{k}\\
      \gcd(h,k)=1}}\omega_{h,k}e^{-\frac{2\pi i hn}{k}}
  \frac{\,d}{\,dn}\frac{\sinh\Big(\frac{\pi}{k}\sqrt{\frac{2}{3}(n-1/24)}\Big)}
  {\sqrt{n-1/24}},
\end{equation}
for all integers $n\ge 1$. In order to obtain Theorem \ref{thm 1}, we use the modular properties of $G$ that arise from \eqref{a} and \eqref{eqn:Pmodular} to obtain a similar exact formula for $c(n)$. We then write $c(n)$ as a main term plus an error term. In order to show that $\sgn(c(n))$ satisfies the claimed periodicity, we require some bounds on $I$-Bessel functions that are used to bound the error term.

\begin{lem}[\cite{Ben}, Lemma 2.4] \label{L:Bessel}\hspace{0cm}
		\begin{enumerate}
			\item For $0\le x<1$, $\kappa\in\ \mathbb{R}$ with $\kappa>-\frac12$, we have
			\begin{equation*}
				I_{\kappa}(x) \le \frac{2^{1-\kappa}x^{\kappa}}{\Gamma(\kappa+1)}.
			\end{equation*}
			\item For $x\geq 1$ and $\kappa\in\mathbb{R}$ with $\kappa>-\frac{1}{2}$, we have
			\begin{equation*}
				I_{\kappa}(x) \le \sqrt{\frac2{\pi x}} e^x.
			\end{equation*}
			\item For $\kappa\geq \frac{1}{2}$ and $x\ge 3$, we have
			\begin{equation*}
				I_{\kappa}(x) > \frac{2^{-\kappa}}{5\sqrt{\pi }\Gamma\left(\kappa+\frac{1}{2}\right)}\frac{e^x}{\sqrt{x}},
			\end{equation*}
		\end{enumerate}
	\end{lem}
\begin{proof}
The proof follows \cite[Lemma 2.4]{Ben}. For part (3), we use the integral representation 
\begin{equation}\label{eqn:Iint}
	I_{\kappa}(x)=\frac{\left(\frac{x}{2}\right)^{\kappa}}{\sqrt{\pi}\Gamma\left(\kappa+\frac{1}{2}\right)} \int_{-1}^1\left(1-u^2\right)^{\kappa-\frac{1}{2}}e^{xu} du.
\end{equation}
We take $u\mapsto 1-u$ and then $u\mapsto \frac{u}{x}$ to obtain 
\begin{align*}
I_{\kappa}(x)&=\frac{\left(\frac{x}{2}\right)^{\kappa}e^x}{\sqrt{\pi}\Gamma\left(\kappa+\frac{1}{2}\right)} \int_{0}^2\left(2u-u^2\right)^{\kappa-\frac{1}{2}}e^{-xu} du\\
&=
\frac{\left(\frac{x}{2}\right)^{\kappa}e^x}{x^{\kappa+\frac{1}{2}}\sqrt{\pi}\Gamma\left(\kappa+\frac{1}{2}\right)} \int_{0}^{2x}\left(2u-\frac{u^2}{x}\right)^{\kappa-\frac{1}{2}}e^{-u} du\\
&\geq \frac{2^{-\kappa}e^x}{\sqrt{\pi x}\Gamma\left(\kappa+\frac{1}{2}\right)}
 \int_{3-\sqrt{6}}^{1}\left(2u-\frac{u^2}{x}\right)^{\kappa-\frac{1}{2}}e^{-u} du\\
&\geq \frac{2^{-\kappa}e^x}{\sqrt{\pi x}\Gamma\left(\kappa+\frac{1}{2}\right)} \int_{3-\sqrt{6}}^{1}e^{-u} du=\frac{2^{-\kappa}e^x\left(e^{\sqrt{6}-3}-e^{-1}\right)}{\sqrt{\pi x}\Gamma\left(\kappa+\frac{1}{2}\right)},
\end{align*}
where $\kappa\geq \frac{1}{2}$, $x\geq 3$,  and $2u-\frac{u^2}{3}\geq 1$ is used in the last inequality. We then bound $e^{\sqrt{6}-3}-e^{-1}>\frac{1}{5}$.
\end{proof}

\section{The Proof of Theorem \ref{thm 1}}\label{sec:mainproof}
\subsection{Modular transformation for the generating function}
In order to state the modularo properties of $G$, set $d_j:=\gcd(j,k)$ and let $h_{j}'$ be the unique solution (we omit $j$ in the notation when $j=1$) to the congruence
\[
h h_{j}'\equiv -1\pmod{\frac{k}{d_j}}.
\]
\begin{prop}\label{Pr 1}
Suppose that $k\in\N$ and let $h\in \mathbb{Z}$ satisfy $\gcd(h, k) = 1.$  Then for all $\mathfrak{R}(z) >0$, we have 
\[
G\left(e^{\frac{2\pi i h}{k} - \frac{2\pi z}{k^2}}\right) = \frac{2\cdot3\cdot z^2}{k^2 d_2 d_3^{\frac{3}{2}}}\frac{\omega_{\frac{2h}{d_2}, \frac{k}{d_2}}^2 \omega_{\frac{3h}{d_3}, \frac{k}{d_3}}^3}{\omega_{h, k}} \exp\left({\frac{\pi (d_2^2 + d_3^2 -1)}{12z}} - {\frac{\pi z}{k^2}}\right)\hat{G}\left(h, k, e^{-\frac{\pi}{3z}}\right), 
\]
\vspace{0.5cm}
where
\begin{multline*}
 \hat{G}\left(h, k, e^{-\frac{\pi}{3z}}\right) := P^{-1}\left(\exp\left({\frac{2\pi i h'}{k} - \frac{2\pi}{z}}\right) \right) P^2\left(\exp\left({\frac{2\pi id_2 h'_{d_2}}{k} - \frac{2\pi d_2^2}{2z}}\right)\right)\\
 \times P^3\left(\exp \left({\frac{2\pi id_3 h'_{d_3}}{k} - \frac{2\pi d_3^2}{3z}}\right)\right).
 \end{multline*}
\end{prop}

\begin{proof}
Since all functions involved in the proposition are holomorphic for $\mathfrak{R}(z) > 0 $, by the identity theorem it suffices to assume that $z>0$ is real.

We write 
\begin{multline}\label{eqn:Grewrite}
 G\left(e^{\frac{2\pi i h}{k} - \frac{2\pi z}{k^2}}\right) = P^{-1}\left(e^{\frac{2\pi i h}{k} - \frac{2\pi z}{k^2}}\right) P^2 \left( \exp \left(\frac{2\pi i (\frac{2h}{d_2})}{\frac{k}{d_2}} - \frac{2\pi (\frac{2z}{d_2^2})}{({\frac{k}{d_2}})^2}\right)\right)\\
 \times P^3 \left( \exp \left(\frac{2\pi i (\frac{3h}{d_3})}{\frac{k}{d_3}} - \frac{2\pi (\frac{3z}{d_3^2})}{({\frac{k}{d_3}})^2}\right)\right)
\end{multline}
and note that $\frac{jh}{d_j}$ and $\frac{k}{d_j}$ are relatively prime by construction. We may hence employ the modular transformation \eqref{eqn:Pmodular} to obtain 
\begin{align*}
  G\left(e^{\frac{2\pi i h}{k} - \frac{2\pi z}{k^2}}\right) = &\frac{1}{\omega_{h, k}}   \left(\frac{z}{k}\right)^{-\frac{1}{2}}   \exp\left({\frac{-\pi}{12z} + \frac{\pi z}{12k^2}}\right) P^{-1}\left(\exp{\frac{2\pi i h'}{k} - \frac{2\pi}{z}}\right)\\
  &\hspace{-1cm}\times  \omega_{\frac{2h}{d_2}, \frac{k}{d_2}}^2 \left( \frac{\frac{2z}{d_2^2}}{\frac{k}{d_2}}\right )
 \exp \left({\frac{2\pi}{12 ( \frac{2z}{d_2^2})} - \frac{2\pi \left(\frac{2z}{d_2^2}\right)}{12 \left(\frac{k}{d_2}\right)^2}} \right) 
 P^2\left(\exp\left({\frac{2\pi id_2 h'_{d_2}}{k} - \frac{2\pi d_2^2}{2z}}\right)\right)
 \\
 &\hspace{-1cm}\times \omega_{\frac{3h}{d_3}, \frac{k}{d_3}}^3\left( \frac{\frac{3z}{d_3^2}}{\frac{k}{d_3}}\right )^{\frac{3}{2}} \exp \left({\frac{3\pi}{12 ( \frac{3z}{d_3^2})} - \frac{3\pi \left(\frac{3z}{d_3^2}\right)}{12 \left(\frac{k}{d_3}\right)^2}} \right)   
 P^3\left(\exp \left({\frac{2\pi id_3 h'_{d_3}}{k} - \frac{2\pi d_3^2}{3z}}\right)\right)
\end{align*}
$$
= \frac{2\cdot{3}^{3/2}\cdot z^2}{k^2 d_2 {d_3}^{\frac{3}{2}}}(\omega_{h, k})^{-1} \left(\omega_{\frac{2h}{d_2}, \frac{k}{d_2}}\right)^2 \left(\omega_{\frac{3h}{d_3}, \frac{k}{d_3}}\right)^3 \exp\left({\frac{\pi (d_2^2 + d_3^2 -1)}{12z}} - {\frac{\pi z}{k^2}}\right) \hat{G}\left(h, k, e^{-\frac{\pi}{3z}}\right).
$$
This completes the proof.
\end{proof}
\subsection{Hardy--Ramanujan--Rademacher expansion for the Fourier coefficients of \texorpdfstring{$G(q)$}{G(q)}}

Based on Rademacher's insights \cite{R partition}, for $n,N\in\N$ we have
\begin{equation}\label{b}
c(n) = \sum_{1\leq k \leq N}\;\; {\sum_{\substack{{0 \leq h<k} \\
{\gcd(h, k) = 1}}}} {\frac{i}{k^2}e^{\frac{-2\pi i h n}{k}}} \int_{z_1}^{z_2} G\left( \exp\left({\frac{2\pi i h}{k} - \frac{2\pi z}{k^2}}\right)\right) \exp \left( \frac{2\pi n z}{k^2}\right) dz.
\end{equation}
Here, $z$ moves along an arc of the circle defined by $C :  |z-\frac{1}{2}| = \frac{1}{2}$ with $\mathfrak{R}(z)>0$. The endpoints of this arc, labeled as $z_1$
and $z_2$ are highlighted in Figure 1 and given by

\begin{equation}\label{c}
 z_1 = \frac{k^2}{k^2 + k_1^2} + i\frac{kk_1}{k^2 + k_1^2} \;\; \;\;\text{and} \;\;\; z_2 = \frac{k^2}{k^2 + k_2^2} + i\frac{kk_2}{k^2 + k_2^2},
\end{equation}

\vspace{0.5cm}
\noindent respectively. Here $h_1/k_1 < h/k < h_2/k_2$ are three consecutive Farey fractions in the Farey sequence of order $N$.
Applying Proposition \ref{Pr 1} to equation \eqref{b} results in 

\begin{multline*}
c(n) = \left(\sum_{\substack{1\leq k \leq N \\  2\nmid k \;\; \text{or}\;\; 3\nmid k}} + \sum_{\substack{1\leq k \leq N \\  2\mid k \;\; \text{and}\;\; 3\mid k}} \right) {\sum_{\substack{{0 \leq h<k} \\
{\gcd(h, k) = 1}}}} \frac{i\cdot2\cdot(3)^{3/2}}{k^4 d_2 d_3^{3/2}} \frac{\omega_{\frac{2h}{d_2}, \frac{k}{d_2}}^2\omega_{\frac{3h}{d_3}, \frac{k}{d_3}}^3}{\omega_{h,k}}
e^{\frac{-2\pi i h n}{k}} 
\\
\times \int_{z_1}^{z_2} z^2 \exp\left( \frac{\pi (d_2^2 + d_3^2 -1)}{12z} +\frac{(24n-1)\pi z}{12k^2}\right) \times \hat{G} \left( h, k, e^{\frac{-\pi}{3z}} \right) \, dz.
\end{multline*}
\qquad $\;\;\;\; =: I + E.$\\
For the first sum $I$, we have  $2\nmid k$ or $3\nmid k$, leading us to three possible scenarios
\[I = \sum_{\substack{1\leq k \leq N \\  2\nmid k \;\; \text{or}\;\; 3\nmid k}}\;\;\; {\sum_{\substack{{0 \leq h<k} \\
{\gcd(h, k) = 1}}}}  \frac{i\cdot2.(3)^{3/2}}{k^4 d_2 d_3^{3/2}} \left( \omega_{h, k}\right)^{-1}\left( \omega_{\frac{2h}{d_2}, \frac{k}{d_2}} \right)^2 \left( \omega_{\frac{3h}{d_3}, \frac{k}{d_3}} \right)^3 e^{\frac{-2\pi i h n}{k}} \]

\[\times \int_{z_1}^{z_2} z^2 \exp\left( \frac{\pi (d_2^2 + d_3^2 -1)}{12z} -\frac{\pi z}{12k^2} +\frac{2\pi n z}{k^2}\right) \times \hat{G} \left( h, k, e^{\frac{-\pi}{3z}} \right)  dz.\] 

\begin{align*}
I = &\left(\sum_{\substack{1\leq k \leq N \\  2\nmid k \; \text{and}\; 3\nmid k}} + \sum_{\substack{1\leq k \leq N \\  2\mid k \; \text{and}\; 3\nmid k}} + \sum_{\substack{1\leq k \leq N \\  2\nmid k \; \text{and}\; 3\mid k}}\right)\;\;\; {\sum_{\substack{{0 \leq h<k} \\
{\gcd(h, k) = 1}}}}  \frac{i\cdot2.(3)^{3/2}}{k^4 d_2 d_3^{3/2}} \left( \omega_{h, k}\right)^{-1}\left( \omega_{\frac{2h}{d_2}, \frac{k}{d_2}} \right)^2 \left( \omega_{\frac{3h}{d_3}, \frac{k}{d_3}} \right)^3 \\
&e^{\frac{-2\pi i h n}{k}}\times \int_{z_1}^{z_2} z^2 \exp\left( \frac{\pi (d_2^2 + d_3^2 -1)}{12z} + \frac{(24n-1)\pi z}{12k^2}\right) \times \hat{G} \left( h, k, e^{\frac{-\pi}{3z}} \right) dz.\\
&= I_1 + I_2 + I_3,
\end{align*}

We first consider $I_1$, writing it as a main term plus an error term.
\begin{prop}\label{prop:I1}
We have 
\[
I_1= 4\cdot3^{3/2} \pi (24n-1)^{-3/2} I_3\left( \frac{\pi}{6}  \sqrt{(24n-1)}\right) + \mathscr{E}_1
\]
where 
\[
\left|\mathscr{E}_1\right|\leq 175,200,000+\frac{\pi^2}{8\sqrt{3}(24n-1)}+\frac{2\sqrt{3}\pi^2}{(24n-1)} I_3\left(\frac{\pi}{30}\sqrt{24n-1}\right).
\]
\end{prop}
We prove Proposition \ref{prop:I1} by spitting the integral into a number of pieces. For the first sum $I_1$, the relation $2\nmid k$ and $3\nmid k$ means $d_2 = \gcd(2, k)= 1$ and $d_3 = \gcd(3, k) = 1$, we now obtain the estimate for $I_1$,

\begin{figure}
    \centering
    \includegraphics[width=0.8\linewidth]{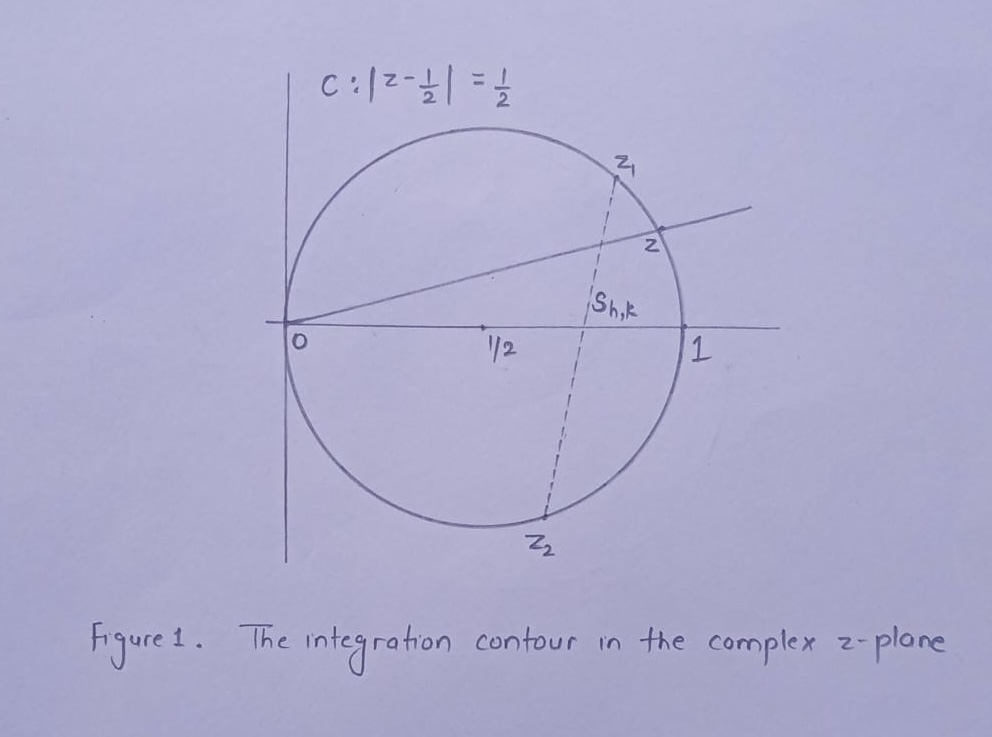}
   
    \label{fig:placeholder}
\end{figure}

\vspace{1cm}

\noindent
\begin{align*}
  I_1 &= \sum_{\substack{1\leq k \leq N \\  2\nmid k \;\; \text{and}\;\; 3\nmid k}}\;\;\; {\sum_{\substack{{0 \leq h<k} \\
{\gcd(h, k) = 1}}}} \frac{ i\cdot2\cdot(3)^{3/2}}{k^4 } \left( \omega_{h, k}\right)^{-1} \left( \omega_{2h, k}\right)^{2} \left( \omega_{3h, k}\right)^{3} e^{\frac{-2\pi i h n}{k}}\\
&\quad \times \int_{z_1}^{z_2} z^2 \exp\left( \frac{\pi}{12z} + \frac{(24n-1)\pi z}{12k^2}\right) \left( 1+\hat{G} \left( h, k, e^{\frac{-\pi}{3z}} \right)- 1\right) dz. \\ 
&=: I_{M_1} + I_{R_1}.
\end{align*}

\noindent
Here $I_{M_1}$ and $I_{R_1}$ come from splitting the last factor of the above equation in the integrand into $1$ and $\left(\hat{G} \left( h, k, e^{\frac{-\pi}{3z}} \right)- 1\right)$.  In the above sum, the path of integration is along the arc of the circle $C$,
the path of integration in $z-$plane can be moved so that  for estimating $I_{R_1}$, we integrate along the chord joining  $z_1$ and $z_2$.  On the chord $s_{h, k},$ from Rademacher \cite[eq.(119.3) and eq.(119.6)]{R analytic}, we have

\begin{equation}\label{Real(z)}
 1\leq \mathfrak{R}\left(\frac{1}{z}\right),  \;\;\;\;  0< \mathfrak{R}(z) < \frac{2k^2}{N^2},     
\end{equation}

\noindent
and from \cite{R partition}, the length of the chord $s_{h,k}$ is 

\begin{equation}\label{s(h,k)}
|s_{h, k} | < \frac{2k}{N+1}.    
\end{equation}
On the chord itself, we have $|z| \leq \max \{ |z_1|, |z_2|\} < \frac{\sqrt{2}k}{N},$ while on $C$ we have $\mathfrak{R}(1/z) = 1$ \cite{Apostol}.
Thus, we have

\begin{align*}
 |I_{R_1}| &\leq \sum_{\substack{1\leq k \leq N \\  2\nmid k \;\; \text{and}\;\; 3\nmid k}}\;\;\; {\sum_{\substack{{0 \leq h<k} \\
{\gcd(h, k) = 1}}}} \frac{2\cdot(3)^{3/2} |s_{h, k}|}{k^4} \\
&\hspace{4.3cm}\times \sup_{z\in s_{h, k} } \left| z^2 \exp\left( \frac{\pi}{12z} + \frac{(24n-1)\pi z}{12k^2}\right) \left( \hat{G} \left( h, k, e^{\frac{-\pi}{3z}} \right) -1 \right)\right|\\
&\leq \sum_{\substack{1\leq k \leq N \\  2\nmid k \;\; \text{and}\;\; 3\nmid k}}\;\;\; {\sum_{\substack{{0 \leq h<k} \\ {\gcd(h, k) = 1}}}} \frac{2\cdot(3)^{3/2}}{k^4} \times \frac{2k}{N+1} \times \left(\frac{4k^4}{N^4} +\frac{1}{4}\right)\times\exp\left(  \frac{(24n-1)\pi \mathfrak{R}(z)}{12k^2}\right) \\
&\hspace{7.5cm} \times \sup_{z\in s_{h, k} } \left|\exp\left(\frac{\pi}{12z}\right)  \left(\hat{G} \left( h, k, e^{\frac{-\pi}{3z}} \right) -1 \right) \right|.
\end{align*}

\noindent
For all $w_1, w_2$ and $w_3 \in \mathbb{C} $ with $|w_1|, |w_2|, |w_3| <1$, it is easy to show that 

\begin{equation}\label{bound w_1 w_2 w_3}
  \left|\left(P^{-1}(w_1) P^2(w_2) P^3(w_3)\right) -1\right| \leq P(|w_1|) P^2(|w_2|)P^3(|w_3|) - 1.  
\end{equation}

\noindent
Using this argument, we obtain the following. 

\begin{align*}
\sup_{z\in s_{h, k} } &\left|\exp\left(\frac{\pi}{12z}\right)  \left(\hat{G} \left( h, k, e^{\frac{-\pi}{3z}} \right) -1 \right) \right| \\
&\leq \sup_{z\in s_{h, k} } \left| \exp\left(\frac{\pi}{12}\mathfrak{R}\left(\frac{1}{z}\right)\right)\left( P \left( e^{-2\pi \mathfrak{R}\left(\frac{1}{z}\right)}\right)P^2 \left( e^{-\pi \mathfrak{R}\left(\frac{1}{z}\right)}\right)P^3 \left( e^{-\frac{2\pi}{3} \mathfrak{R}\left(\frac{1}{z}\right)}\right) - 1 \right)\right|\\
&\leq e^{\frac{\pi}{12}}\left(P \left( e^{-2\pi}\right)P^2 \left( e^{-\pi}\right)P^3 \left( e^{-\frac{2\pi}{3}}\right) - 1 \right).
 \end{align*}

\begin{rem}
    Although the prefactor $\exp\left(\frac{\pi}{12}\mathfrak{R}\left(\frac{1}{z}\right)\right)$ in the above inequality grows rapidly as $\mathfrak{R}\left(\frac{1}{z}\right) \to \infty$, it does not dominate the supremum. This is because the term in parentheses, $\left( P \left( e^{-2\pi \mathfrak{R}\left(\frac{1}{z}\right)}\right)P^2 \left( e^{-\pi \mathfrak{R}\left(\frac{1}{z}\right)}\right)P^3 \left( e^{-\frac{2\pi}{3} \mathfrak{R}\left(\frac{1}{z}\right)}\right) - 1 \right)$ vanishes exponentially faster. Subtracting 1 and multiplying by the external exponential factor yields a bounded quantity. Hence, the supremum is not dictated by the leading exponential but by the lower-order terms.

\end{rem}
 

\begin{lem}\label{bnd p(n)}
Based on above notation, we have
\begin{align*}
    \left|P\left(e^{-2\pi}\right) P^2\left(e^{-\pi}\right) P^3\left(e^{\frac{-2\pi}{3}}\right)\right| \leq 2.
\end{align*}
\end{lem}

\begin{proof}
Using the well-known estimate of \cite{Apostol}.
\begin{equation}\label{p(n)}
    p(n) \leq e^{\pi \sqrt{\frac{2n}{3}}},
\end{equation}
for $n\geq n_0$ sufficiently large we obtain 
$$ p(n) e^{-\pi \alpha n}\leq e^{\pi \sqrt{\frac{2 n}{3}} -\pi\alpha n} \leq e^{-\frac{\pi \alpha n}{2}}.$$
Hence 
\[
1\leq P(e^{-\alpha \pi }) = \sum_{n\geq 0} p(n)e^{-\alpha\pi n}\leq \sum_{n=0}^{n_0-1}p(n) e^{-\pi \alpha n} + \sum_{n=n_0}^{\infty} e^{-\frac{\pi \alpha n}{2}}=\sum_{n=0}^{n_0-1}p(n) e^{-\pi \alpha n} + \frac{e^{-\frac{\pi\alpha n_0}{2}} }{1-e^{-\frac{\pi \alpha}{2}}}.
\]
Using this, we obtain 
\[
1\leq P(e^{-2\pi})\leq 1.1,\qquad 1\leq P^2(e^{-\pi})\leq 1.1,\quad\text{ and }1\leq P^3\left(e^{-\frac{2\pi}{3}}\right)\leq 1.6.
\]
Therefore
\[
\left|P\left(e^{-2\pi}\right) P^2\left(e^{-\pi}\right) P^3\left(e^{\frac{-2\pi}{3}}\right)\right| \leq 2.\qedhere
\]
\end{proof}

\noindent
By using Lemma \ref{bnd p(n)} and equation \eqref{Real(z)}, the estimate bound for $E_2$ is

\[ |I_{R_1}| \leq  \sum_{\substack{1\leq k \leq N \\ 2\nmid k \;\; \text{and}\;\; 3\nmid k}}\;\;\; {\sum_{\substack{{0 \leq h<k} \\ {\gcd(h, k) = 1}}}} \frac{17.3^{3/2}}{k^3 N} \exp\left( \frac{\pi}{12} + \frac{(24n-1)\pi }{6N^2}\right). \]

\vspace{1em}
\noindent
Using the fact that $\phi(k) \leq k$ and fixing $N = \lceil\sqrt{n}\rceil$, we have

\[|I_{R_1}| \leq  \sum_{\substack{1\leq k \leq N \\ 2\nmid k \;\; \text{and}\;\; 3\nmid k}} \frac{17\cdot3^{3/2}}{k^2 \sqrt{n}}\exp \left(\frac{\pi}{12} + 4\pi \right) \leq {17\cdot3^{3/2}}\zeta(2)\exp \left(\frac{49\pi}{12}  \right). \]

\noindent
After all calculations, we have
\begin{equation}\label{eqn:E2Bound}
    |I_{R_1}| \leq 54,200,000.
\end{equation}

\noindent
To evaluate $I_{M_1}$ we split the integral into two parts $I_{MM_1}$ and $I_{ME_1}$, as indicated below; the path of integration of the first part is the whole circle $C$, traversed from 0 to 0 in negative direction, while in the second part the arc of the circle is traversed from $z_1$ to $z_2$ (which itself is split into a difference of two integrals with respective paths of integration along the arc starting at 0). Specifically, we have

\begin{align*}
  I_{M_1} = &\sum_{\substack{1\leq k \leq N \\  2\nmid k \;\; \text{and}\;\; 3\nmid k}}\;\;\; {\sum_{\substack{{0 \leq h<k} \\
{\gcd(h, k) = 1}}}} \frac{ i\cdot2 \cdot(3)^{3/2}}{k^4 } \left( \omega_{h, k}\right)^{-1} \left( \omega_{2h, k}\right)^{2} \left( \omega_{3h, k}\right)^{3} e^{\frac{-2\pi i h n}{k}}\\
&\times \left(\int_{C}-\left(\int_{0}^{z_1}+ \int_{0}^{z_2}\right)\right) z^2 \exp\left( \frac{\pi}{12z} + \frac{(24n-1)\pi z}{12k^2}\right)  dz. \\ 
&= I_{MM_1} + I_{ME_1}.
\end{align*}

\noindent
We estimate the second part, $I_{ME_1}$, first. Combining \eqref{Real(z)} with the fact that $\mathfrak{R}\left(\frac{1}{z}\right) = 1$ on the circle and the length of the arc from $0$ to $z_j$ is less than 

\[\frac{\pi}{2}|z_j|\leq \frac{\pi k}{\sqrt{2}N},\]
we obtain that
\begin{align*}
    |I_{ME_1}| &\leq 2\sum_{\substack{1\leq k \leq N \\  2\nmid k \;\; \text{and}\;\; 3\nmid k}}\;\;\; {\sum_{\substack{{0 \leq h<k} \\ {\gcd(h, k) = 1}}}} \frac{2\cdot3^{3/2}}{k^4} \times \frac{\pi k}{\sqrt{2} N}\times \left(\frac{4k^4}{N^4} + \frac{1}{4}\right) \times\exp \left( \frac{\pi}{12} + \frac{(24n -1)\pi}{6N^2}\right)\\
&\leq \sum_{\substack{1\leq k \leq N \\  2\nmid k \;\; \text{and}\;\; 3\nmid k}} \frac{2^{-1/2}\cdot 3^{3/2}\cdot 17\pi}{k^2\sqrt{n}} \exp \left( \frac{\pi}{12} + 4\pi\right) \leq 2^{-1/2}\cdot  3^{3/2}\cdot 17\pi \zeta(2)\exp \left( \frac{49\pi}{12} \right). 
\end{align*}
Upon completing all the calculations, we obtain
\begin{equation}\label{eqn:E12Bound}
 |I_{ME_1}| \leq 121,000,000. 
\end{equation}
Next, we determine the estimated bound for 
\begin{equation}\label{integral}
I_{MM_1} =  i\cdot2 \cdot(3)^{3/2}\;\hspace{-.7cm}\sum_{\substack{1\leq k \leq N \\ \;\; 2\nmid k \;\; \text{and}\;\; 3\nmid k}}\hspace{-.2cm} \frac{H_k(n) }{k^4} \int_{C} z^2 \exp\left( \frac{\pi}{12z} +\frac{(24n-1)\pi z}{12k^2}\right) dz,
\end{equation}
where
\begin{equation*}
    H_k(n) := {\sum_{\substack{{0 \leq h<k} \\
{\gcd(h, k) = 1}}}}  \left( \omega_{h, k}\right)^{-1} \left( \omega_{2h, k}\right)^{2} \left( \omega_{3h, k}\right)^{3} e^{\frac{-2\pi i h n}{k}}.
\end{equation*}
The integral in \eqref{integral} can be computed in terms of Bessel functions. The change of variables
\begin{equation*}
 w = \frac{1}{z}, \;\;\;\; dz = -w^{-2}dw,
\end{equation*}
yields
\begin{equation}\label{E_11}
 I_{MM_1} =  i\cdot2\cdot(3)^{3/2}\; \hspace{-.7cm}\sum_{\substack{1\leq k \leq N \\  2\nmid k \;\; \text{and}\;\; 3\nmid k}} \hspace{-.2cm}\frac{H_k(n)}{k^4}  L_0,
\end{equation}
where 
\begin{equation}\label{L0}
L_0 := -\int_{1-\infty i}^{1+\infty i} w^{-4} \exp \left( \frac{\pi w}{12} +\frac{(24n-1) \pi}{12 k^2 w}\right) dw.
\end{equation}
We now put  $ t = \frac{\pi w}{12}$ in \eqref{L0} to rewrite the above integral as 
\begin{equation}\label{eq L_0}
    L_0 = -\left(\frac{\pi}{12}\right)^3 \int_{c-\infty i}^ {c+ \infty i} t^{-4} \exp \left( t + \frac{z^2}{4t}\right) dt
\end{equation}
with $ c = \frac{\pi}{12}$ and $ z = \frac{\pi}{6k} \sqrt{(24n-1)}$. From \cite[p. 181]{Watson}, we find the formula 

\begin{equation}\label{watson formula}
 I_\nu(z) = \frac{(\frac{1}{2}z)^{\nu}}{2 \pi i} \int_{c-\infty i}^{c+\infty i} t^{-\nu-1} \exp\left( t + \frac{z^2}{4t}\right) dt, \;\;\; \;\;\; (\text{if} \;\;c>0, \mathfrak{R}(\nu) >0).   
\end{equation}
Here $ I_\nu(z) = i^{-\nu}J_\nu(z) (iz)$ is the $I$-Bessel function. Comparing \eqref{eq L_0} and \eqref{watson formula}, we have
\begin{align}
\nonumber L_0 &= -\left(\frac{\pi}{12}\right)^3 \times 2\pi i \times \left( \frac{\pi}{12k} \sqrt{(24n-1)}\right)^{-3} I_3\left( \frac{\pi}{6k}  \sqrt{(24n-1)}\right)\\
\label{L_00}
&= - 2\pi i k^3 (24n-1)^{-3/2} I_3\left( \frac{\pi}{6k}  \sqrt{(24n-1)}\right).
\end{align}
We put  \eqref{L_00} into \eqref{E_11} to obtain
\begin{equation}\label{Semi E_{11}}
 I_{MM_1} = 4\cdot3^{3/2} \pi  \sum_{\substack{1\leq k \leq N \\  2\nmid k \;\; \text{and}\;\; 3\nmid k}} \frac{1}{k}(24n-1)^{-3/2} H_k(n) I_3\left( \frac{\pi}{6k}  \sqrt{(24n-1)}\right).
\end{equation}
We split the equation \eqref{Semi E_{11}} into a main term
\begin{equation}\label{eqn:MainI1}
    I_{MMM_1} = 4\cdot3^{3/2} \pi (24n-1)^{-3/2}\; H_1(n) I_3\left( \frac{\pi}{6}  \sqrt{(24n-1)}\right)
\end{equation}
and an error term 
\begin{equation*}
    I_{MME_1} = 4\cdot3^{3/2} \pi  \sum_{\substack{2\leq k \leq N \\  2\nmid k \;\; \text{and}\;\; 3\nmid k}} \frac{1}{k}(24n-1)^{-3/2} H_k(n) I_3\left( \frac{\pi}{6k}  \sqrt{(24n-1)}\right).
\end{equation*}
\begin{lem}\label{lem:E11Bound}
 We have 
\begin{equation*}
| I_{MME_1}| \leq \frac{{\pi}^2}{8\sqrt{3}(24n-1)} +  \frac{2 \sqrt{3}{\pi}^2}{(24n-1)}I_3\left( \frac{\pi}{30} \sqrt{(24n-1)}\right).       
\end{equation*}
\end{lem}
\begin{proof}
By taking the absolute value inside the sum in \eqref{Semi E_{11}} and noting that $|H_k(n)|\leq k$, we conclude that
\begin{equation}\label{E_111}
    | I_{MME_1}| \leq \frac{4\cdot(3)^{3/2}\pi}{(24n-1)^{3/2}} \sum_{k\geq 2}I_3\left( \frac{\pi}{6k}  \sqrt{(24n-1)}\right).
\end{equation}
We split the sum over $k$ in \eqref{E_111} into two parts: one for large $k$ and one for small $k$. Applying Lemma \ref{L:Bessel} (1), we deduce that the contribution to \eqref{E_111} from $k > \frac{\pi}{6}\sqrt{(24n-1)}$ is bounded by
\begin{equation}\label{E_1111}
\frac{{\pi}^4}{144\sqrt{3}} \sum_{k>\frac{\pi}{6}\sqrt{(24n-1)}} k^{-3} \leq \frac{{\pi}^4}{144\sqrt{3}} \int_{\frac{\pi}{6}\sqrt{(24n-1)}}^{\infty} x^{-3} dx = \frac{{\pi}^2}{8\sqrt{3}(24n-1)}.
\end{equation}
By approximating the Bessel function for the term with $k = 5$, we deduce that the contribution from $k \leq \frac{\pi}{6}\sqrt{(24n-1)}$ is bounded by
\begin{equation}\label{E_11111}
\frac{4\cdot(3)^{3/2}\pi}{(24n-1)^{3/2}} \sum_{5\leq k \leq \frac{\pi}{6}\sqrt{(24n-1)}} I_3\left( \frac{\pi}{6k}  \sqrt{(24n-1)}\right) \leq \frac{2 \sqrt{3}{\pi}^2}{(24n-1)}I_3\left( \frac{\pi}{30} \sqrt{(24n-1)}\right).
\end{equation}
Adding \eqref{E_1111} and \eqref{E_11111} gives the claim.\qedhere
\end{proof}



\begin{proof}[Proof of Proposition \ref{prop:I1}]
We obtain Proposition \ref{prop:I1} by taking the main term \eqref{eqn:MainI1} (note that $H_1(n)=1$) and  adding together the error terms coming from \eqref{eqn:E2Bound}, \eqref{eqn:E12Bound}, and Lemma \ref{lem:E11Bound}. 
\end{proof}

The arguments for $I_2$ and $I_3$ are similar, so most of the calculations are omitted. 
\begin{prop}\label{prop:I2}
We have 
\[
I_2=3^{\frac{3}{2}} \cdot 8\pi(24n-1)^{-\frac{3}{2}} (-1)^n I_3\left(\frac{\pi}{6}\sqrt{24n-1}\right) + \mathscr{E}_2
\]
with 
\[
\left|\mathscr{E}_2\right|\leq 64,900,000+4,367,640+\frac{\pi^2}{2\sqrt{3}(24n-1)} + \frac{8\sqrt{3}\pi^2}{24n-1}I_3\left(\frac{\pi}{12}\sqrt{24n-1}\right).
\]
\end{prop}
\begin{proof}
As in the proof of Proposition \ref{prop:I1}, we split 
\begin{align*}
I_2 =& \sum_{\substack{1\leq k \leq N \\  2\mid k \; \text{and}\; 3\nmid k}}\;\; {\sum_{\substack{{0 \leq h<k} \\
{\gcd(h, k) = 1}}}} \frac{i\cdot3^{3/2}}{k^4} \left( \omega_{h, k}\right)^{-1}\left( \omega_{h, \frac{k}{2}} \right)^2 \left( \omega_{{3h}, {k}} \right)^3 e^{\frac{-2\pi i h n}{k}}\\
& \times \int_{z_1}^{z_2} z^2 \exp\left( \frac{\pi}{3z} + \frac{(24n-1)\pi z}{12k^2}\right) \times \left( 1 + \hat{G} \left( h, k, e^{\frac{-\pi}{3z}} \right) - 1 \right)dz.\\
&=I_{M_2} + I_{R_2}.
\end{align*}
We further split $I_{M_2}=I_{MM_2}+I_{ME_2}$ and then $I_{MM_2}=I_{MMM_2}+I_{MME_2}$ analogously to the proof of Proposition \ref{prop:I1}, with the main term $I_{MMM_2}$ coming from the $k=2$ term of $I_{MM_2}$. A tedious but straightforward calculation, following the steps from the proof of Proposition \ref{prop:I1}, then yields
\begin{align*}
\left|I_{R_2}\right|&\leq 64,900,000,\\
\left|I_{ME_2}\right|&\leq 4,367,640,\\
\left|I_{MME_2}\right|&\leq \frac{\pi^2}{2\sqrt{3}(24n-1)} + \frac{8\sqrt{3}\pi^2}{24n-1}I_3\left(\frac{\pi}{12}\sqrt{24n-1}\right).\qedhere
\end{align*}
\end{proof}

\begin{prop}\label{prop:I3}
We have 
\[
I_3= 36\pi(24n-1)^{-\frac{3}{2}}B_3(n)I_3\left(\frac{\pi}{6}\sqrt{24n-1}\right)+ \mathscr{E}_3
\]
with 
\[
B_3(n)=\begin{cases}\sqrt{3}&\text{if }n\equiv 0,3\pmod{6},\\
-\sqrt{3}&\text{if } n\equiv 1,4\pmod{6},\\
0&\text{if }n\equiv 2,5\pmod{6},
\end{cases}
\]
and 
\begin{multline*}
\left|\mathscr{E}_3\right|\leq 115,700,000+257,978,900+\frac{\pi^2}{8(24n-1)}+\frac{6\pi^2}{24n-1}I_3\left(\frac{\pi}{12}\sqrt{24n-1}\right)\\
+\frac{\pi^2}{72(24n-1)}+\frac{2\pi^2}{3(24n-1)} I_3\left(\frac{\pi}{18}\sqrt{24n-1}\right).
\end{multline*}
\end{prop}
\begin{proof}
As in the proof of Proposition \ref{prop:I1}, we split 
\begin{align*}
I_3 =& \sum_{\substack{1\leq k \leq N \\  2\nmid k \; \text{and}\; 3/ k}}\;\; {\sum_{\substack{{0 \leq h<k} \\
{\gcd(h, k) = 1}}}} \frac{2i}{k^4} \left( \omega_{h, k}\right)^{-1}\left( \omega_{2h, k} \right)^2 \left( \omega_{{h}, \frac{k}{3}} \right)^3 e^{\frac{-2\pi i h n}{k}}\\
& \times \int_{z_1}^{z_2} z^2 \exp\left( \frac{3\pi}{4z} + \frac{(24n-1)\pi z}{12k^2}\right) \times \left( 1 + 3e^{\frac{-2\pi}{3z}} + \hat{G} \left( h, k, e^{\frac{-\pi}{3z}} \right) - 1 - 3e^{\frac{-2\pi}{3z}} \right)dz.\\ 
& = I_{M_3} +I_{M_3}'+ I_{R_3},
\end{align*}
Here there are two main terms $I_{M_3}$ coming from $1$ and $I_{M_3}'$ coming from $3e^{-\frac{2\pi}{3z}}$. As in the proof of Proposition \ref{prop:I1}, we continue to split these as $I_{M_3}=I_{MM_3}+I_{ME_3}$ (and $I_{M_3}'=I_{MM_3}'+I_{ME_3}'$) and then $I_{MM_3}=I_{MMM_3}+I_{MME_3}$ (while the corresponding $I_{MM_3}'$ is an error term).

We then compute 
\[
I_{MMM_3}=36\pi(24n-1)^{-\frac{3}{2}}B_3(n)I_3\left(\frac{\pi}{6}\sqrt{24n-1}\right)
\]
and bound the error terms 
\begin{align*}
\left|I_{R_3}\right|&\leq 115,700,000,\\
\left|I_{ME_3}\right|+\left|I_{ME_3}'\right|&\leq 257,978,900,\\
\left|I_{MME_3}\right|&\leq \frac{\pi^2}{8(24n-1)}+\frac{6\pi^2}{24n-1}I_3\left(\frac{\pi}{12}\sqrt{24n-1}\right),\\
\left|I_{MM_3}'\right|&\leq \frac{\pi^2}{72(24n-1)}+\frac{2\pi^2}{3(24n-1)} I_3\left(\frac{\pi}{18}\sqrt{24n-1}\right).\qedhere
\end{align*}
\end{proof}
We bound $|E|$ entirely analogously.
\begin{prop}\label{prop:E}
We have 
\begin{multline*}
\left|E\right|\leq  127,000,000+ 282,100,000+ \frac{\pi^2}{6(24n-1)} + \frac{8\pi^2}{24n-1}I_3\left(\frac{\pi}{6\sqrt{3}}\sqrt{24n-1}\right)\\
+\frac{\pi^2}{108(24n-1)} + \frac{8\pi^2}{9(24n-1)}I_3\left(\frac{\pi}{18}\sqrt{24n-1}\right).
\end{multline*}
\end{prop}
\begin{proof}
As in the proof of Proposition \ref{prop:I3}, we split 
\begin{align*}
E =& \sum_{\substack{1\leq k \leq N \\  2\mid k \; \text{and}\; 3\mid k}}\;\; {\sum_{\substack{{0 \leq h<k} \\
{\gcd(h, k) = 1}}}} \frac{i}{k^4} \left( \omega_{h, k}\right)^{-1}\left( \omega_{h, \frac{k}{2}} \right)^2 \left( \omega_{{h}, \frac{k}{3}} \right)^3 e^{\frac{-2\pi i h n}{k}}\\
& \times \int_{z_1}^{z_2} z^2 \exp\left( \frac{\pi}{z} + \frac{(24n-1)\pi z}{12k^2}\right) \times \left( 1 + 3e^{\frac{-2\pi}{3z}} + \hat{G} \left( h, k, e^{\frac{-\pi}{3z}} \right) - 1 - 3e^{\frac{-2\pi}{3z}} \right)dz\\
&=E_M +E_{M}' + E_R,    
\end{align*}
where $E_M$ (resp. $E_M'$) is the contribution from $1$ (resp. $3e^{-\frac{2\pi}{3z}}$). As in the proof of Proposition \ref{prop:I3}, we split $E_M=E_{MM}+E_{ME}$ and $E_{M}'=E_{MM}'+E_{ME}'$. In this case, there is no main term, and an analogous calculation shows that 
\begin{align*}
\left|E_R\right|&\leq 127,000,000,\\
\left|E_{ME}\right|+\left|E_{ME}'\right|&\leq 282,100,000,\\
\left|E_{MM}\right|&\leq \frac{\pi^2}{6(24n-1)} + \frac{8\pi^2}{24n-1}I_3\left(\frac{\pi}{6\sqrt{3}}\sqrt{24n-1}\right),\\
\left|E_{MM}'\right|&\leq \frac{\pi^2}{108(24n-1)} + \frac{8\pi^2}{9(24n-1)}I_3\left(\frac{\pi}{18}\sqrt{24n-1}\right).\qedhere
\end{align*}
\end{proof}

We are now ready to prove Theorem \ref{thm 1} for $n\not\equiv 4\pmod{6}$.
\begin{prop}
For $n\not\equiv 4\pmod{6}$, we have 
\[
\sgn(c(n))=\begin{cases} 1&\text{if }n\equiv 0,2,3\pmod{6},\\
-1&\text{if }n\equiv 1,5\pmod{6}.
\end{cases}
\]
\end{prop}
\begin{proof}
We combine Proposition \ref{prop:I1}, Proposition \ref{prop:I2}, Proposition \ref{prop:I3}, and Proposition \ref{prop:E} to conclude that 
\[
c(n)=  
\left(4\cdot3^{3/2} +(-1)^n 8\cdot 3^{\frac{3}{2}} + 36B_3(n)\right)
\frac{\pi}{(24n-1)^{3/2}} I_3\left( \frac{\pi}{6}  \sqrt{(24n-1)}\right) + \mathscr{E}
\]
with 
\begin{multline}\label{eqn:scrEBound}
\left|\mathscr{E}\right|\leq 175,200,000+\frac{\pi^2}{8\sqrt{3}(24n-1)}+\frac{2\sqrt{3}\pi^2}{(24n-1)} I_3\left(\frac{\pi}{30}\sqrt{24n-1}\right)\\
+64,900,000+4,367,640+\frac{\pi^2}{2\sqrt{3}(24n-1)} + \frac{8\sqrt{3}\pi^2}{24n-1}I_3\left(\frac{\pi}{12}\sqrt{24n-1}\right)\\
+115,700,000+257,978,900+\frac{\pi^2}{8(24n-1)}+\frac{6\pi^2}{24n-1}I_3\left(\frac{\pi}{12}\sqrt{24n-1}\right)\\
+\frac{\pi^2}{72(24n-1)}+\frac{2\pi^2}{3(24n-1)} I_3\left(\frac{\pi}{18}\sqrt{24n-1}\right)\\
+127,000,000+ 282,100,000+ \frac{\pi^2}{6(24n-1)} + \frac{8\pi^2}{24n-1}I_3\left(\frac{\pi}{6\sqrt{3}}\sqrt{24n-1}\right)\\
+\frac{\pi^2}{108(24n-1)} + \frac{8\pi^2}{9(24n-1)}I_3\left(\frac{\pi}{18}\sqrt{24n-1}\right)\\
=1,027,246,540+\left(\frac{4}{27}+\frac{5}{8\sqrt{3}}\right)\frac{\pi^2}{24n-1}+ \frac{2\sqrt{3}\pi^2}{(24n-1)} I_3\left(\frac{\pi}{30}\sqrt{24n-1}\right)\\
+\frac{(8\sqrt{3}+6)\pi^2}{24n-1}I_3\left(\frac{\pi}{12}\sqrt{24n-1}\right) + \frac{14\pi^2}{9(24n-1)}I_3\left(\frac{\pi}{18}\sqrt{24n-1}\right)\\
+ \frac{8\pi^2}{24n-1}I_3\left(\frac{\pi}{6\sqrt{3}}\sqrt{24n-1}\right).
\end{multline}
The main term is 
\begin{equation}\label{eqn:MainOverall}
\frac{\pi\sqrt{3}}{(24n-1)^{3/2}} I_3\left( \frac{\pi}{6}  \sqrt{(24n-1)}\right) \times \begin{cases}
72&\text{if }n\equiv 0\pmod{6},\\
-48&\text{if }n\equiv 1\pmod{6},\\
36&\text{if }n\equiv 2\pmod{6},\\
24&\text{if }n\equiv 3\pmod{6},\\
-12&\text{if }n\equiv 5\pmod{6}.
\end{cases}
\end{equation}
Hence we obtain the claim if \eqref{eqn:scrEBound} is smaller than the absolute value of \eqref{eqn:MainOverall}.

The claim hence holds as long as 
\begin{multline*}
1\geq \frac{1,027,246,540(24n-1)^{\frac{3}{2}}}{12\sqrt{3}\pi I_3\left(\frac{\pi}{6}\sqrt{24n-1}\right)} +\frac{\left(\frac{4}{27}+\frac{5}{8\sqrt{3}}\right)\pi^2\sqrt{24n-1}}{12\sqrt{3}\pi I_3\left(\frac{\pi}{6}\sqrt{24n-1}\right)}\\
+\frac{2\sqrt{3}\pi^2\sqrt{24n-1}}{12\sqrt{3}\pi} \frac{I_3\left(\frac{\pi}{30}\sqrt{24n-1}\right)}{I_3\left(\frac{\pi}{6}\sqrt{24n-1}\right)}+\frac{(8\sqrt{3}+6)\pi^2\sqrt{24n-1}}{12\sqrt{3}\pi} \frac{I_3\left(\frac{\pi}{12}\sqrt{24n-1}\right)}{I_3\left(\frac{\pi}{6}\sqrt{24n-1}\right)}\\
+ \frac{14\pi^2\sqrt{24n-1}}{9\cdot12\sqrt{3}\pi}\frac{I_3\left(\frac{\pi}{18}\sqrt{24n-1}\right)}{I_3\left(\frac{\pi}{6}\sqrt{24n-1}\right)}+ \frac{8\pi^2\sqrt{24n-1}}{12\sqrt{3}\pi} \frac{I_3\left(\frac{\pi}{6\sqrt{3}}\sqrt{24n-1}\right)}{I_3\left(\frac{\pi}{6}\sqrt{24n-1}\right)}.
\end{multline*}
Since $n\geq 4$, we may use Lemma \ref{L:Bessel} to bound the Bessel functions, and hence we see that the claim holds if
\begin{multline*}
1\geq \frac{75(1,027,246,540)\sqrt{\pi}(24n-1)^{\frac{7}{4}}}{36\sqrt{2}e^{\frac{\pi}{6}\sqrt{24n-1}}} +\frac{75\left(\frac{4}{27}+\frac{5}{8\sqrt{3}}\right)\pi^{\frac{5}{2}}(24n-1)^{\frac{3}{4}}}{36\sqrt{2}e^{\frac{\pi}{6}\sqrt{24n-1}}}\\
+\frac{75\pi^{\frac{3}{2}}\sqrt{10}\sqrt{24n-1}}{12\sqrt{3}e^{\frac{2\pi}{15}\sqrt{24n-1}}}+\frac{150(8\sqrt{3}+6)\pi^{\frac{3}{2}}\sqrt{24n-1}}{12\sqrt{3}e^{\frac{\pi}{12}\sqrt{24n-1}}}+ \frac{75\cdot 14\sqrt{6}\pi^{\frac{3}{2}}\sqrt{24n-1}}{108\sqrt{3}e^{\frac{\pi}{9}\sqrt{24n-1}}}\\
+ \frac{75\cdot 8\sqrt{2} \pi^{\frac{3}{2}}\sqrt{24n-1}}{12\cdot3^{\frac{1}{4}} e^{\frac{(\sqrt{3}-1)\pi}{6\sqrt{3}}\sqrt{24n-1}}}.
\end{multline*}
Setting $x:=\sqrt{24n-1}$, this holds if 
\[
f(x):=\frac{2682213774 x^\frac{7}{2}}{e^{\frac{\pi}{6}x}} + \frac{14 x^{\frac{3}{2}}}{e^{\frac{\pi}{6}x}}+\frac{64x}{e^{\frac{2\pi x}{15}}}+\frac{798 x}{e^{\frac{\pi}{12}x}}+\frac{70 x}{e^{\frac{\pi}{9}x}}+\frac{300 x}{e^{\frac{(\sqrt{3}-1)\pi}{6\sqrt{3}}x}}\leq 1.
\]
Computing 
\begin{multline*}
f'(x):=\frac{2682213774 x^\frac{7}{2}}{e^{\frac{\pi}{6}x}}\left(-\frac{\pi}{6}+\frac{7}{2x}\right) + \frac{14 x^{\frac{3}{2}}}{e^{\frac{\pi}{6}x}}\left(-\frac{\pi}{6}+\frac{3}{2x}\right)+\frac{64x}{e^{\frac{2\pi x}{15}}}\left(-\frac{2\pi}{15}+\frac{1}{x}\right)\\
+\frac{798 x}{e^{\frac{\pi}{12}x}}\left(-\frac{\pi}{12}+\frac{1}{x}\right)+\frac{70 x}{e^{\frac{\pi}{9}x}}\left(-\frac{\pi}{9}+\frac{1}{x}\right)+\frac{300 x}{e^{\frac{(\sqrt{3}-1)\pi}{6\sqrt{3}}x}}\left(-\frac{(\sqrt{3}-1)\pi}{6\sqrt{3}}+\frac{1}{x}\right),
\end{multline*}
one easily checks that $f'(x)<0$ for $x\geq 8$. Since $f(70)<1$, we conclude that $f(x)<1$ for $x\geq 70$, or, equivalently, $n\geq 205$. We check the remaining cases with a computer.\qedhere
\end{proof}

\section{Vanishing of Coefficients for \texorpdfstring{$n \equiv 4 \pmod 6$}{n=4(mod 6)}}\label{sec:vanish}
To complete the proof of Theorem \ref{thm 1}, we need to verify that the remaining coefficients vanish.
\begin{lem}
Writing 
\[
F(z)=q^{-\frac{1}{2}}\sum_{n\geq 0}c(n)q^n,
\]
we have,
\[c(n) =0 \text{ for } n \equiv 4 \pmod 6.\]
\end{lem}
\begin{proof}
Set
\[ g(z)=F(2z) = \sum_{n\geq 0} c(n) q^{2n-1}.\]
We the claim that $ g\big| S_{12, 7}=0.$ We have 
\begin{align*}
    g\big| S_{12, 7}(z) &= \sum_{\substack{{n \geq 0}\\ {n\equiv 4 \pmod {5}}}} c(n) q^{2n-1} \\
    &= \frac{1}{12} {\sum_{\substack{{n \geq 0}\\ {j\pmod {12}}}}} c(n) q^{2n-1} e^{\frac{2\pi i (2n-1)j}{12}} e^{\frac{-2\pi i 7j}{12}}\\
    &= \frac{1}{12}  {\sum_ {j\pmod {12}}} e^{\frac{-2\pi i 7j}{12}} \sum_{n\geq 0} c^*(n) e^{2\pi i (2n-1) (z + \frac{j}{12})}\\
&= \frac{1}{12} \sum_{j\pmod {12}} e^{\frac{-2\pi i 7j}{12}} g \left(z+ \frac{j}{12}\right).
\end{align*}

Suppose that the maximal order of a pole of $g$ at any cusp is $\leq J$. Then  \[ h_J := \Delta^J g\big|_{S_{12, 7}} \]
is a holomorphic modular form of weight $12J - 2$  in $\Gamma \subseteq SL_2(\mathbb{Z)}$, where $\Gamma = \Gamma_0(48) \cap \Gamma_1(12).$ We then use the valence formula and check vanishing for the first $(J-\frac{1}{6})[\operatorname{SL}_2(\Z):\Gamma]$ coefficients to verify that $h_J=0$, and hence so is $g|S_{12,7}.$
\end{proof}

\section{Combinatorial interpretation of the eta-quotient}\label{sec:combinatorics}

We now prove Theorem \ref{thm:combinatorics}.
\begin{proof}[Proof of Theorem \ref{thm:combinatorics}]
By Theorem \ref{thm 1}, in order to show the claim, we only need to show that $c(n)=r_{\operatorname{e}}(n)-r_{\operatorname{o}}(n)$. To show this, we rewrite the eta-quotient as 
\begin{equation}\label{eqn:G1-2-3product}
G_{(1,-2,-3)}(q)=q^{\frac{1}{2}}F(\tau)= \frac{q^{\frac{1}{8}}}{\frac{\eta(2\tau)^2}{\eta(\tau)}} \cdot \left(\frac{q^{\frac{1}{8}}}{\eta(3\tau)}\right)^3.
\end{equation}
Using \cite[Theorem 1.60]{Ken}, we have 
\begin{align*}
\frac{q^{\frac{1}{8}}}{\frac{\eta(2\tau)^2}{\eta(\tau)}}& =\frac{1}{\sum_{n\geq 0} q^{\frac{(2n+1)^2-1}{8}}}=\frac{1}{1+\sum_{n\geq 1} q^{T_n}} = \sum_{j=0}^{\infty}(-1)^j\left(\sum_{n\geq 1} q^{T_n}\right)^j\\
&=1+ \sum_{n_1=1}^{\infty} q^{n_1}\sum_{j=1}^{n_1} \sum_{\substack{(x_1,\dots,x_j)\in\N^j\\ T_{x_1}+\dots+T_{x_j}=n_1}}(-1)^j.
\end{align*}
and 
\begin{align*}
\frac{q^{\frac{1}{8}}}{\eta(3\tau)}&= \frac{1}{\sum_{w=-\infty}^{\infty}(-1)^w q^{\frac{(6w+1)^2-1}{8}}} =  \frac{1}{1+\sum_{w\in\Z\setminus\{0\}}(-1)^w q^{3P_5(w)}}\\
&=\sum_{j=0}^{\infty}\left(\sum_{w\in\Z\setminus\{0\}}(-1)^w q^{3P_5(w)}\right)^j\\
&=1+\sum_{n=1}^{\infty}q^{3n}\sum_{j=1}^n\sum_{\substack{(w_1,\dots,w_{j})\in (\Z\setminus\{0\})^j\\ P_5(w_1)+\dots+P_5(w_j)=n}} (-1)^{j}(-1)^{w_1+w_2+\dots+w_{j}}
\end{align*}
Plugging into \eqref{eqn:G1-2-3product} and combining powers of $q$ yields the claim.
\end{proof}

\end{document}